\documentclass{amsart}
\usepackage[cp1251]{inputenc}
\usepackage[english]{babel}
\usepackage{amsmath}
\usepackage{amssymb}
\usepackage{amsfonts}
\usepackage{a4wide}

\newtheorem{lemma}{Lemma}
\newtheorem{theorem}{Theorem}

\newtheorem{proposition}{Proposition}

\begin{document}

\title[On energy functionals for elliptic systems]
{On energy functionals\\for second-order elliptic systems\\with constant
coefficients}

\author{A.~Bagapsh, K.~Fedorovskiy}

\address{Astamur Bagapsh${}^{1,2}$:\newline %
\hphantom{iii} ${1)}$~Federal Research Center \newline %
\hphantom{iii} `Computer Science and Control' \newline %
\hphantom{iii} of the Russian Academy of Sciences,\newline %
\hphantom{iii} Moscow, Russia;\newline %
\hphantom{iii} ${2)}$~Moscow Center for Fundamental \newline %
\hphantom{iii} and Applied Mathematics, \newline %
\hphantom{iii} Lomonosov Moscow State University, \newline %
\hphantom{iii} Moscow, Russia}

\email{a.bagapsh@gmail.com}

\address{Konstantint Fedorovskiy${}^{1,2}$:\newline %
\hphantom{iii} ${1)}$~Faculty of Mechanics and Mathematics, \newline %
\hphantom{iii} Lomonosov Moscow State University, \newline %
\hphantom{iii} Moscow, Russia; \newline %
\hphantom{iii} ${2)}$~Moscow Center for Fundamental \newline %
\hphantom{iii} and Applied Mathematics, \newline %
\hphantom{iii} Lomonosov Moscow State University, \newline %
\hphantom{iii} Moscow, Russia}

\email{kfedorovs@yandex.ru}

\thanks{\rm The work is partially supported by the Theoretical Physics and
Mathematics Advancement Foundation `BASIS'}

%%%%%%%%%%%%%%%%%%%%%%%%%%%%%%%%%%%%%%%%%%%%%%%%%%%%%%%%%%%%%%%%
\maketitle
{
\small
\begin{quote}
\noindent{\bf Abstract. } We consider the Dirichlet problem for second-order
elliptic systems with constant coefficients. We prove that non-reducible
strongly elliptic systems of this type do not admits non-negatively defined
energy functionals of the form
$$
f\mapsto\int_{D}\varPhi(u_x,v_x,u_y,v_y)\,dxdy,
$$
where $D$ is the domain where the problem we are interested in is considered,
$\varPhi$ is some quadratic form in $\mathbb R^4$, and $f=u+iv$ is a function
in the complex variable. The proof is based on reducing the system under
consideration to a special (canonical) form, when the differential operator
defining this system is represented as a perturbation of the Laplace operator
with respect to two small real parameters (the canonical parameters of the
system under consideration).
\medskip

\end{quote}
}

\section{Introduction, problem statement and some background}

In the paper we are dealing with the Dirichlet problem (in its classical
setting) for second-order elliptic systems in $\mathbb R^2$ with constant
coefficients. For real-valued functions $u$ defined in $\mathbb R^2$ we will
denote by $u_x$, $u_y$, $u_{xx}$ etc., their partial derivatives with respect
to the corresponding variables. Furthermore, we will use the differential
operators $\partial_x=\partial/\partial x$ and $\partial_y=\partial/\partial
y$. In what follows the symbol $M_k(\mathbb R)$ will state for the space of
all real $k\times k$-matrices (here $k>0$ is an integer), while the symbol
$A^t$ will stand for the transposed matrix to $A$.

We are interested in the question whether the energy functionals of the form
\begin{equation}\label{eq:enfunc-0}
f\mapsto\int_{D}\varPhi(u_x,u_y,v_x,v_y)\,dxdy,
\end{equation}
do exist for systems under consideration, where $D$ is a domain where the
problem in considered, $f=u+iv$ is some function in the complex variable, and
$\varPhi$ is some non-negatively determined quadratic form in $\mathbb R^4$.
The question on whether such energy functional exists is mainly motivated by
the Dirichlet problem for real-valued harmonic functions, since for these
functions such functional exists and has the form
$\int_{D}((u_x)^2+(u_y)^2)\,dxdy$. For general systems under consideration
and, in particular, for second-order elliptic equations with constant complex
coefficients the question about existence of positively determined energy
functionals still open in the general case (as well a question on general
solvability of the corresponding Dirichlet problem in general bounded simply
connected domains).

In what follows we will identify the points $z=(x,y)$ in the plane $\mathbb
R^2$ with the complex numbers $z=x+iy$. Moreover, we will identify the pairs
of functions $u$ and $v$, defined in $\mathbb R^2$ and taking real values,
with the complex-valued function $f(z)=u(x,y)+iv(x,y)$, and vice-versa.
Furthermore, the symbol $f$ will mean, if necessary, the vector $(u,v)^t$.

Let $A,B,C\in M_2(\mathbb R)$. Define the differential operator
\begin{equation}\label{eq:gen-oper}
\mathcal L=A\partial_{xx}+2B\partial_{xy}+C\partial_{yy},
\end{equation}
where, as usual, $\partial_xf=u_x+iu_y$ and $\partial_yf=u_y+iv_y$. In other
words, $\mathcal Lf=\widetilde{u}+i\widetilde{v}$, where $\widetilde{u}$ and
$\widetilde{v}$ are defined as follows:
$$
\begin{pmatrix}\widetilde{u}\\\widetilde{v}\end{pmatrix}
=A\kern1pt\begin{pmatrix}u_{xx}\\v_{xx}\end{pmatrix}
+2B\kern1pt\begin{pmatrix}u_{xy}\\v_{xy}\end{pmatrix}
+C\kern1pt\begin{pmatrix}u_{yy}\\v_{yy}\end{pmatrix}.
$$

We consider the homogeneous system of equations of the form
\begin{equation}\label{eq:gen-syst}
\mathcal Lf=0.
\end{equation}
The most important particular case that we are interested in during this work
is the system given by the matrices $A,B,C\in M_2^\sharp$, where
$$
M_2^\sharp=\left\{A\in M_2(\mathbb R) : A=\begin{pmatrix}x&-y\\y&x\end{pmatrix}\right\}.
$$
Observe that the set $M_2^\sharp$ considered with the standard operations of
matrix summation and multiplications is the field isomorphic to the field
$\mathbb C$ of complex numbers. Thus the system \eqref{eq:gen-syst} with
$A,B,C\in M_2^\sharp$ is equivalent to one second-order equation with
constant \emph{complex} coefficients on the (complex-valued) function
$f=u+iv$, that is to the equation of the form
\begin{equation}\label{eq:gen-compl-equation}
af_{xx}+2bf_{xy}+cf_{yy}=0,
\end{equation}
where $a,b,c\in\mathbb C$, $f_x=\partial_xf$ and $f_y=\partial_yf$ (as in the
real-valued case). The systems corresponding to the equations of the form
\eqref{eq:gen-compl-equation}, are often called \emph{skew-symmetric},
despite on a clear inaccuracy of this term.

Recall that an ellipticity of the system \eqref{eq:gen-syst} means that the
corresponding \emph{characteristic form}
\begin{equation}\label{eq:char-form}
\mathcal F(\xi,\eta)=\det(A\xi^2+2B\xi\eta+C\eta^2)
\end{equation}
with real $\xi$ and $\eta$ vanishes only for $\xi=\eta=0$ (see, for instance,
\cite{petr1939}).

The ellipticity condition for the equation \eqref{eq:gen-compl-equation} is
equivalent to the fact the the corresponding \emph{symbol}
$a\xi^2+2b\xi\eta+c\eta^2$ with real $\xi$ and $\eta$ is also vanished only
for $\xi=\eta=0$. The latter condition is equivalent to the fact that the
roots of the corresponding \emph{characteristic equation}
$a\lambda^2+2b\lambda+c=0$ are not real.

In the general case it follows from the ellipticity condition that $\det
A\neq 0$ and $\det C\neq 0$ (otherwise $\mathcal F(t,0)=0$ and $\mathcal
F(0,t)=0$ for $t\neq0$, respectively). Since $\mathcal
F(\xi,\eta)=\eta^4\det(A\lambda^2+2B\lambda+C)$, where $\lambda=\xi/\eta$,
the ellipticity of the system \eqref{eq:gen-syst} is equivalent to the fact
that $\det A\neq0$ and all roots of the fourth-order equation (with real
coefficients)
$$
\det(A\lambda^2+2B\lambda+C)=0
$$
are not real. In such a case we have two pairs of \emph{complex conjugate}
roots that we denote by $\lambda_1$, $\overline\lambda_1$, $\lambda_2$ and
$\overline\lambda_2$.

The Dirichlet problem for the system \eqref{eq:gen-syst} is as follows: given
a bounded simply connected domain $D$ in the plane and a continuous function
$h$ on the boundary $\partial D$ of $D$, to find a function $f$ of class
$C^2(D)\cup C(\overline{D})$ such that $\mathcal Lf=0$ in $D$ and
$f|_{\partial D}=h$ (in view of the ellipticity of $\mathcal L$ it is
sufficient to state the problem to find $f\in C(\overline{D})$ satisfying the
equation $\mathcal Lf=0$ in $D$ in the distributional sense). The problem is
appeared natural to describe such domains $D$, where the Dirichlet problem is
solvable for every given continuous function $h$ on $\partial D$. Domains
satisfying this property are called \emph{$\mathcal L$-regular}.

In the problem on description of $\mathcal L$-regular domains the concept
appears natural of an equivalence of systems under consideration. Two systems
of the form \eqref{eq:gen-syst} are \emph{equivalent}, if they can be reduced
one to another by the use of the following \emph{admissible transformations}:
not degenerate real linear changes of variables and sought-for functions, and
not degenerate real linear combinations of equations. In what follows we will
call them admissible transformation of the \emph{first}, \emph{second} and
\emph{third} kinds, respectively. If two systems of the form
\eqref{eq:gen-syst} set by the operators $\mathcal L_1$ and $\mathcal L_2$
are equivalent, and if a domain $D$ is $\mathcal L_1$-regular, then the
domain obtained from $D$ using the corresponding linear transformation is
$\mathcal L_2$-regular. Such notion of system equivalence was introduced, for
instance, in \cite{hlw1985b}.

One of the most simple case is when the system \eqref{eq:gen-syst} can be
reduced by transformation of three kinds mentioned above to the system wuth
upper triangular mtrices $A$, $B$ and $C$. Such system is called
\emph{reducible}. This term is related with the fact that the system
\eqref{eq:gen-syst} with upper triangular matrices splits into two
independent elliptic equations with constant \emph{real} coefficients, the
first one of which is homogeneous, while the second one has nonzero
right-hand part. It can be shown that the first this equation is equivalent
to the Laplace equation, while the second one is equivalent to the Poisson
equation for harmonic functions.

For harmonic functions (that is for the system given by the Laplace operator
$\Delta$, $\Delta u=u_{xx}+u_{yy}$) the result is well known that any bounded
simply connected domain is $\Delta$-regular (this breakthrough result was
proved by Lebesgue \cite{leb1907} in 1907, and one of the important
ingredient of the proof was the fact that for the Laplace operator the energy
functional of the form under consideration do exist). Thus, for reducible
systems it is known the complete description of regular domains, as well an
answer to the question on existence of energy functional of the desired form.
Notice that the class of reducible systems is the almost one class, for which
the complete answers to both questions under consideration were obtained
(except such systems the complete answer is known for the system
corresponding to the anisotropic Lame equation).

In what follows we are dealing with non-reducible systems. One of the most
known and important case of such systems are skew-symmetric systems that
arise from equations of the form \eqref{eq:gen-compl-equation}.

Studying non-reducible systems in the context of the Dirichlet problem as
well in the context of the problem on existence of energy functionals of the
form \eqref{eq:enfunc-0} it is natural to pay attention to the notion of
\emph{strong ellipticity}.

The next definition of strong ellipticity was introduced by Vishik in
\cite{vis1951}: the system \eqref{eq:gen-syst} is said to be strongly
elliptic, if for every $\xi,\eta\in\mathbb R$ the following matrix
$$
A_+\xi^2+2B_+\xi\eta+C_+\eta^2,
$$
is positively determined, where $X_+=(X+X^t)/2$ for $X\in M_2(\mathbb R)$.
Notice that every strongly elliptic system in the sense of Vishik is
elliptic, but the opposite is, in general, not the case.

In \cite{hlw1985b} the notion of strong ellipticity was introduced in a
slightly different way. Namely, the system \eqref{eq:gen-syst} is strongly
elliptic, if
$$
\det(\alpha A+2\beta B+\gamma C)\neq0
$$
for all real $\alpha$, $\beta$ and $\gamma$ satisfying the condition
$\beta^2-\alpha\gamma<0$. It can be shown that both these definitions of
strong ellipticity are equivalent by modulo of system equivalence mentioned
above.

For further considerations we need yet another properties of systems
\eqref{eq:gen-syst}. One says that a \emph{strongly elliptic} system of the
form \eqref{eq:gen-syst} is \emph{symmetrizable}, if it can be reduced by
admissible transformations to the system with symmetric matrices $A$, $B$ and
$C$ such that the block $4\times 4$-matrix
$$
\begin{pmatrix} A & B \\ B & C \end{pmatrix}
$$
is positively determined. In the opposite case the strongly elliptic system
\eqref{eq:gen-syst} under consideration is said to be
\emph{non-symmetrizable}.

The paper is organized as follows. In Section~2 we consider the canonical
forms of the system \eqref{eq:gen-syst} and discuss the meaning of
corresponding canonical parameters. In Section~3 we formulate and prove the
main result of the paper (see Theorem~\ref{thm:main}) that gives a criterion
in order that the system \eqref{eq:gen-syst} admits (non-negatively
determined) energy functional of the form \eqref{eq:enfunc-0}. This criterion
is stated in terms of canonical parameters of the system \eqref{eq:gen-syst}.
In particular, it follows from Theorem~\ref{thm:main} that for strongly
elliptic equation \eqref{eq:gen-compl-equation} an energy functional of the
desired form exists if and only if this equation has real (up to a common
complex multiplier) coefficients, that is when this equation is equivalent to
the Laplace one. Thus the standard proof of the fact that any bounded simply
connected domain in the plane is regular with respect to the Dirichlet
problem for harmonic functions cannot be generalized to the case of general
strongly elliptic equations with constant complex coefficients.

\section{Canonical form of second-order elliptic systems with constant coefficients}

It is convenient to begin the study of our questions by reducing the system
\eqref{eq:gen-syst} to one of the canonical forms using admissible
transformations, which are linear changes of variables and sought-for
functions, and changes of the equations by their suitable linear
combinations. Let us emphasize that most of the facts presenting in this
section are not new. They are known and can be found, for example, in
\cite{hlw1985b} or \cite{vv1997} (see also \cite{bf2017}). However, the
complex canonical form presented at the end of this section is recently
appeared in works by the authors. In what follows the symbol $z$ will mean
not only the complex variable $z=x+iy$ and the corresponding point in the
plane, but also the column-vector $(x,y)^t$.

The next lemma is verified by the direct differentiation.
\begin{lemma}\label{lem:transform-1}
Let $A$, $B$ and $C$ be the matrices $A$, $B$ and $C$ that set the
differential operator $\mathcal L$ of the form \eqref{eq:gen-oper}. Then $A$,
$B$ and $C$ are changed as follows under admissible transformations:
\begin{enumerate}
\item[1)]
Let $\zeta=Tz$, $\zeta=\xi+i\eta$, be the change of coordinates with the
matrix $T\in M_2(\mathbb R)$, $\det{T}\neq0$. Then
$$
\mathcal Lf=A'f_{\xi\xi}+2B'f_{\xi\eta}+C'f_{\eta\eta},
$$
where
\begin{align*}
A'=&\begin{pmatrix}t_{11}&t_{12}\end{pmatrix}
\begin{pmatrix}A&B\\B&C\end{pmatrix}\begin{pmatrix}t_{11}\\t_{12}\end{pmatrix}
=t_{11}^2A+2t_{11}t_{12}B+t_{12}^2C,\\
B'=&\begin{pmatrix}t_{11}&t_{12}\end{pmatrix}
\begin{pmatrix}A&B\\B&C\end{pmatrix}\begin{pmatrix}t_{21}\\t_{22}\end{pmatrix}
=t_{11}t_{21}A+(t_{11}t_{22}+t_{12}t_{21})B+t_{21}t_{22}C,\\
C'=&\begin{pmatrix}t_{21}&t_{22}\end{pmatrix}
\begin{pmatrix}A&B\\B&C\end{pmatrix}\begin{pmatrix}t_{21}\\t_{22}\end{pmatrix}
=t_{21}^2A+2t_{21}t_{22}B+t_{22}^2C,
\end{align*}
and where $t_{jk}$, $j,k=1,2$, are the elements of $T$.
\item[2)]
Let $\varphi=Qf$ be the transformation of the sought-for functions given
by the matrix $Q\in M_2(\mathbb R)$, $\det{Q}\neq0$. Then
$$
\mathcal Lf=A'\varphi_{xx}+2B'\varphi_{xy}+C'\varphi_{yy},
$$
where $A'=AQ$, $B'=BQ$ and $C'=CQ$.
\item[3)]
A linear combination of equations of the system \eqref{eq:gen-syst}
defined by the matrix $P\in M_2(\mathbb R)$, $\det{P}\neq0$, leads to the
system of equations set by the operator
$$
\mathcal L'f=A'f_{xx}+2B'f_{xy}+C'f_{yy},
$$
where $A'=PA$, $B'=PB$ and $C'=PC$.
\end{enumerate}
\end{lemma}

The next lemma also can be verified by the direct computation.

\begin{lemma}\label{lem:char-trans}
Let the elliptic system \eqref{eq:gen-syst} has the characteristic form
$\mathcal F(\xi,\eta)$ with roots $\lambda_1$, $\overline{\lambda_1}$,
$\lambda_2$, $\overline{\lambda_2}$. Then \textup(in terms of notations
introduced in Lemma~\ref{lem:transform-1}\textup) the following properties
takes place.
\begin{enumerate}
\item[1)]
The linear change of variables $\zeta=Tz$ leads to the system with the
characteristic form
$$
\det{A'}\big(\xi^2+|\lambda'_1|^2\eta^2\big)\big(\xi^2+|\lambda'_2|^2\eta^2\big),
$$
where $A'$ is defined in the statement \textup(1\textup) of
Lemma~\ref{lem:transform-1}, and $\lambda'_k=\varLambda_T(\lambda_k)$ for
$k=1,2$, where
\begin{equation}\label{eq:moebius}
\varLambda_T(\lambda):=\frac{t_{22}\lambda-t_{21}}{-t_{12}\lambda+t_{11}}.
\end{equation}
\item[2)]
The linear change of sought-for functions $\varphi=Qf$ leads to the
system with the characteristic form $q\mathcal F$ with $q=\det{Q}$.
\item[3)]
The linear combination of equations of the initial system defined by the
matrix $P$ leads to the system with the characteristic form $p\mathcal F$
with $p=\det{P}$.
\end{enumerate}
\end{lemma}

Using these two technical lemmata we are able to prove the first statement
about canonical form of the systems under consideration.

\begin{proposition}\label{thm:canon-1}
Every non-reducible elliptic system of the form \eqref{eq:gen-syst} can be
reduced by admissible transformations to the system set by the operator
$$
\mathcal L^{1}_{\kappa,\lambda}=A\partial_{xx}+2B\partial_{xy}+C\partial_{yy},
$$
where parameters $\kappa$ and $\lambda$ are such that $\kappa\in(0,1]$ and
$\lambda\in[-\kappa,\kappa]\setminus\{0,\kappa^2\}$, and where
\begin{equation}\label{canon-matr-1}
A=\begin{pmatrix}1&0\\0&1\end{pmatrix},\qquad
B=\begin{pmatrix}0&1\\-\frac14(1-\lambda)(1-\kappa^2\big/\lambda)&0\end{pmatrix},\qquad
C=\begin{pmatrix}\lambda&0\\0&\kappa^2\big/\lambda\end{pmatrix}.
\end{equation}
In such a case strong ellipticity of the initial system is equivalent to the
fact that $\lambda>0$.
\end{proposition}

\begin{proof}[Scheme of the proof]
Let $\mathcal L$ from \eqref{eq:gen-syst} is set by the matrices $A$, $B$ and
$C$. In order to reduce $\mathcal L$ to the desired canonical form we
simplify $A$, $B$ and $C$ in four steps.

\smallskip
\emph{Step~1. Simplification of the characteristic form.}\quad Since the set
of characteristic roots of the initial system consists of two pairs of
complex conjugate numbers (and since one may assume that $\lambda_1$ and
$\lambda_2$ belongs to the upper half-plane), then there exists a M\"obius
(linear-fractional) transformation, which maps the upper half-plane onto
itself and moving all characteristic roots of the initial system to the
imaginary axis. Such transformation $\varLambda$ may be found using the
following conditions:
\begin{equation}\label{eq:meob-cond}
\varLambda(\lambda_1)=\kappa i,\qquad\varLambda(\lambda_2)=i,
\end{equation}
where $\kappa\in\mathbb R$, $\kappa\neq0$, is an unknown parameter that needs
to be determined.

In the case, where $\lambda_1=\lambda_2$, the desired $\varLambda$ is a
composition of shift and dilation. In this case $\kappa=1$. For
$\lambda_1\neq\lambda_2$ let us observe that the points $\lambda_1$,
$\lambda_2$, $\overline\lambda_1$ and $\overline\lambda_2$ lie on some circle
orthogonal to the real axis. Let $\zeta_*$ and $\zeta_{**}$ are two points
where this circle intersects the real line. The function
$$
\varLambda(\zeta)=\rho\frac{\zeta-\zeta_*}{\zeta-\zeta_{**}}
$$
maps the circle under consideration onto the imaginary line, and the
parameter $\rho$ is determined by the condition $\varLambda(\lambda_2)=i$.
After that the relation $\varLambda(\lambda_1)=\kappa i$ gives the value of
$\kappa$. If $\kappa>1$, then instead of $\varLambda$ we will use the
composition of this $\varLambda$ and dilation to $\kappa$ times. Therefore we
obtain a M\"obius transformation
$$
\varLambda(\zeta)=\frac{a\zeta+b}{c\zeta+d}
$$
that possesses the properties $\varLambda(\lambda_1)=\kappa i$,
$\kappa\in(0,1]$, and $\varLambda(\lambda_2)=i$.

Apply the change of variable given by the matrix
$$
T_1=\begin{pmatrix}d&-c\\-b&a\end{pmatrix}.
$$
Doing this we pass from the operator $\mathcal L$ to the operator $\mathcal
L_1$ set by $A_1$, $B_1$ and $C_1$ defined in the first statement of
Lemma~\ref{lem:transform-1}. According to Lemma~\ref{lem:char-trans}, the new
system (given by $\mathcal L_1$) has the characteristic form
$$
\mathcal F_1(\xi,\eta)=\det A_1\cdot(\xi^2+\eta^2)(\xi^2+\kappa^2\eta^2).
$$
Notice, that $\det{A_1}\neq0$ and $\det{C_1}\neq0$ in view of ellipticity of
the initial system.

\smallskip
\emph{Step~2. Diagonalization of $A_1$.}\quad Apply the third-kind
transformation (linear combination of equations) given by the matrix
$A_1^{-1}$ to the system corresponding to $\mathcal L_1$. We obtain the
system set by the operator $\mathcal L_2$ with the matrices
$A_2=A_1^{-1}A_1=I$, $B_2=A_1^{-1}B_1$ and $C_2=A_1^{-1}C_1$, where $I$ is
the identity matrix. For this system the characteristic form is $\mathcal
F_2(\xi,\eta)=(\xi^2+\eta^2)(\xi^2+\kappa^2\eta^2)$.

\smallskip
\emph{Step~3. Diagonalization of $C_2$.} Let $C_3$ be the Jordan canonical
form of $C_2$. We have $C_3=PC_2P^{-1}$, where $P$ is some suitable
non-degenerate matrix. Such transformation of $C_2$ to $C_3$ corresponds to
consequent admissible transformations of the second and third kinds defined
by $P$. Let $A_3=I$ and $B_3=PB_2P^{-1}$. So, we pass from the system
corresponding to $\mathcal L_2$ to the system defined by the operator
$\mathcal L_3$ with the matrices $A_3$, $B_3$ and $C_3$. The characteristic
form does not change under such a transformation, thus for the new system we
have
\begin{equation}\label{eq:char3-1}
\mathcal F_3(\xi,\eta)=(\xi^2+\eta^2)(\xi^2+\kappa^2\eta^2)=
\xi^4+(1+\kappa^2)\xi^2\eta^2+\kappa^2\eta^4.
\end{equation}

Notice, that $C_3$ may have one of the following forms:
\begin{enumerate}
\item[1)]
$C_3=\begin{pmatrix}\lambda&0\\0&\mu\end{pmatrix}$, where $\lambda$ and
$\mu$ are real (simple) eigenvalues of the matrix $C_2$, which we assume
to be such that $|\lambda|\leqslant|\mu|$;
\item[2)]
$C_3=\begin{pmatrix}\lambda&1\\0&\lambda\end{pmatrix}$, where $\lambda$
is the real eigenvalue of $C_2$ of order two;
\item[3)]
$C_3=\begin{pmatrix}\lambda&-\mu\\\mu&\lambda\end{pmatrix}$, where
$\lambda\pm i\mu$ are two complex conjugate eigenvalues of $C_2$.
\end{enumerate}
In all these cases the eigenvalues of $C_2$ differ from zero because $C_2$ is
non-degenerate. Let
$$
B_3=\begin{pmatrix}b_1&b_2\\b_3&b_4\end{pmatrix}.
$$
In the fist case mentioned above we can specify $B_3$ and $C_3$ comparing
\eqref{eq:char3-1} with the explicit expression for the characteristic form
$\mathcal F_3$ in terms of the elements of $A_3$, $B_3$ and $C_3$:
\begin{align*}
\mathcal F_3(\xi,\eta)&=
\det\begin{pmatrix}\xi^2+2b_1\xi\eta+\lambda\eta^2 & 2b_2\xi\eta\notag \\
2b_3\xi\eta & \xi^2+2b_4\xi\eta+\lambda\mu^2\end{pmatrix}\\
&=\xi^4+2(b_1+b_4)\xi^3\eta+(\lambda+\mu+4b_1b_4-4b_2 b_3)\xi^2\eta^2+2(b_1\mu+b_4\lambda)\xi\eta^3+\lambda\mu\eta^4.
\end{align*}
In the case where $\lambda\neq\mu$ we have
\begin{equation}\label{eq:matr-1a}
A_3=I,\qquad B_3=\begin{pmatrix}0&b_2\\b_3&0\end{pmatrix},\qquad
C_3=\begin{pmatrix}\lambda&0\\0&\kappa^2/\lambda\end{pmatrix},
\end{equation}
where $b_2b_3=-\frac14(1-\lambda)(1-\kappa^2/\lambda)$, while in the opposite
case $\lambda=\mu$ we have
\begin{equation}\label{eq:matr-1b}
A_3=I,\qquad B_3=\begin{pmatrix}b_1&b_2\\b_3&-b_1\end{pmatrix},\qquad
C_3=\begin{pmatrix}\pm\kappa&0\\0&\pm\kappa\end{pmatrix},
\end{equation}
where $b_1^2+b_2b_3=-\frac14(1\mp\kappa)^2$.

In the second case we conclude that $B_3$ and $C_3$ are upper-triangular
matrices (in this case we are arguing by the same way and comparing
\eqref{eq:char3-1} with its explicit expression via elements of $B_3$ and
$C_3$). It means that the system corresponding to $\mathcal L_3$ is reducible
in such a case.

Finally, in the third case we consider the system of conditions that appears
after equalizing the corresponding coefficients in \eqref{eq:char3-1} and in
explicit expression for $\mathcal F_3$ via elements of $B_3$ and $C_3$. It
turns out that this system is inconsistent.

Thus in the case of non-reducible systems only the first form of the Jordan
canonical form of $C_2$ may take place. In this case the matrices of the
corresponding differential operator can be reduced to the form
\eqref{eq:matr-1a} or \eqref{eq:matr-1b}.

\smallskip
\emph{Step~4. Exclusion of superfluous parameters.}\quad Let $A_3$, $B_3$ and
$C_3$ are defined in \eqref{eq:matr-1a}, and let
$P=\mathop{\mathrm{diag}}(b_2,1)$ (i.e. $P$ is the corresponding diagonal
matrix). Then $A_4=P^{-1}A_3P$, $B_4=P^{-1}B_3P$ and $C_4=P^{-1}A_3P$ are
such that
\begin{equation}\label{eq:matr-a}
A_4=I,\qquad B_4=\begin{pmatrix}0&1\\-\frac14(1-\lambda)\big(1-\kappa^2\big/\lambda\big)&0\end{pmatrix},\qquad
C_4=\begin{pmatrix}\lambda&0\\0&\kappa^2\big/\lambda\end{pmatrix},
\end{equation}
and we can pass from the operator with $A_3$, $B_3$ and $C_3$ to the operator
set by $A_4$, $B_4$ and $C_4$ using admissible transformations of the second
and third kind given by $P$ and $P^{-1}$, respectively.

In the case where $A_3$, $B_3$ and $C_3$ are defined by \eqref{eq:matr-1b},
they also can be reduced to matrices of the form $A_4$, $B_4$ and $C_4$ (with
$\lambda=\pm\kappa$) by conjugation with a certain suitable non-degenerate
matrix for which we do not present an explicit form in order to avoid some
bulk computations.

\smallskip
\emph{Refinement of the set of possible values of $\lambda$.}\quad Let us
find the set of possible values of $\lambda$ for non-reducible elliptic
systems whose matrices are reduced to the form \eqref{eq:matr-1a} or
\eqref{eq:matr-1b}. First of all let us recall that since
$|\lambda|\leqslant|\mu|$ and $\lambda\mu=\kappa^2$, then
$\lambda\in[-\kappa,\kappa]$. Furthermore, all matrices in \eqref{eq:matr-1a}
simultaneously become triangle if and only if $b_3=0$. Then, since
$b_2b_3=-\frac14(1-\lambda)(1-\kappa^2/\lambda)$ we obtain that
$\lambda=\kappa^2$ (notice that the value $\lambda=1$ is not admissible
because of $|\lambda|\leqslant|\mu|$). Moreover, all matrices in
\eqref{eq:matr-1b} simultaneously become triangle if and only if $b_3=0$,
which gives $b_1=-\frac14(1\mp\kappa)^2$, i.e. $b_1=0$ and
$\lambda=\kappa=1$. Combining all these observations we obtain that
$\lambda\in[-\kappa,\kappa]\setminus\{0,\kappa^2\}$. Moreover, it can be
directly verified that the strong ellipticity property of the system under
consideration is equivalent to the fact that $\lambda>0$. Indeed, it is
enough to observer that $\det(C_1-\lambda A_1)=0$ (since $\lambda$ is the
eigenvalue of the matrix $C_2=A_1^{-1}C_1$) and to use the formulae for $A_1$
and $C_1$ obtained in Lemma~\ref{lem:transform-1}.
\end{proof}

In what follows it will be convenient to use yet another canonical
representation for the system \eqref{eq:gen-syst}, which is related with the
Cauchy--Riemann operator. Recall, that the Cauchy--Riemann operator is the
differential operator
$$
\overline\partial=\frac{\partial}{\partial\overline{z}}=
\frac12\big(\partial_x+i\partial_y\big)
$$
Together with $\overline\partial$ we will use the operator
$$
\partial=\frac{\partial}{\partial z}=\frac12\big(\partial_x-i\partial_y\big)
$$
Let us recall that the kernel of the operator $\overline\partial$ in some
domain $D\subset\mathbb C$ is the space of all holomorphic functions in $D$,
while the kernel of $\partial$ in $D$ is the space of all antiholomorphic
functions in $D$ (of course, both kernels are considered in the space of
continuous functions in $D$).

Let $\mathcal L$ be the operator of the form \eqref{eq:gen-oper} with the
canonical parameters $\kappa\in(0,1]$ and $\lambda\in[-\kappa,\kappa]$,
$\lambda\neq0,\kappa^2$. We put
$$
\tau=\frac{1-\kappa}{1+\kappa},\qquad \sigma=\frac{\kappa-\lambda}{\kappa+\lambda}.
$$
Let $\lambda>0$ (this case corresponds to the case where the system
\eqref{eq:gen-syst} given by $\mathcal L$ is strong elliptic). Then
$|\sigma|<1$. Define the operator
\begin{equation}\label{eq:se-oper}
\mathcal L_{\tau,\sigma}=(\partial\overline\partial+\tau\partial^2)f+
\sigma(\tau\partial\overline\partial+\partial^2)\overline{f}=0.
\end{equation}
Let now $\lambda<0$, that is the system~\eqref{eq:gen-syst} set by $\mathcal
L$ is not strongly elliptic. In such a case $|\sigma|>1$. For
$\lambda=-\kappa$ we put $\sigma=\infty$. Setting $s=1/\sigma$ we define
$\mathcal L_{\tau,\sigma}$ in this case as follows
\begin{equation}\label{eq:we-oper}
\mathcal L_{\tau,\sigma}= (\overline\partial{}^2+\tau\partial\overline\partial)f+
s(\tau\overline\partial{}^2+\partial\overline\partial)\overline{f}=0.
\end{equation}
Notice that $\mathcal L_{\tau,\sigma}$ for $|\sigma|<1$ may be regarded as a
perturbation of the Laplace operator by a pair of `small' parameters $\tau$
and $\sigma$, while $\mathcal L_{\tau,\sigma}$ for $|\sigma|>1$ may regarded
as a perturbation of the Bitsadze operator $\overline\partial^2$ by small
parameters $\tau$ and $s=1/\sigma$.

\begin{proposition}
Let $\mathcal L$ be the strongly elliptic operator of the form
\eqref{eq:gen-oper}. Then it can be reduced by admissible transformations to
the form $\mathcal L_{\tau,\sigma}$ with $|\sigma|<1$. In the case where
$\mathcal L$ is not strongly elliptic, it can be reduced by admissible
transformations to the form $\mathcal L_{\tau,\sigma}$ with $|\sigma|>1$.

In particular, every strongly elliptic operator of the form
\eqref{eq:gen-oper} may be reduced to the form $\mathcal L_{\tau,0}$, while
every operator of the form \eqref{eq:gen-oper}, which is not strongly
elliptic, may be reduced to the form $\mathcal L_{\tau,\infty}$.
\end{proposition}

\begin{proof}
First of all we observe that every non-reducible elliptic system
\eqref{eq:gen-syst} given by $\mathcal L$ of the form \eqref{eq:gen-oper},
may be rewritten in the form of a single equation to the function $f=u+iv$:
\begin{multline}\label{eq:compl-canon-1}
(1-\kappa)(\kappa+\lambda)\partial^2 f(z)+(1+\kappa)(\kappa+\lambda)\partial\overline\partial f(z)+\\
+(1+\kappa)(\kappa-\lambda)\partial^2\overline{f(z)}+(1-\kappa)(\kappa-\lambda)\partial\overline\partial\,\overline{f(z)}=0,
\end{multline}
where $\kappa$ and $\lambda$ are defined for $\mathcal L$ in
Proposition~\ref{thm:canon-1}. Indeed, let us continue transformation of
matrices $A$, $B$ and $C$ from $\mathcal L$, which was begun in the proof of
Proposition~\ref{thm:canon-1}.

At the \emph{first step} we multiply $A_4$, $B_4$ and $C_4$ from
\eqref{eq:matr-a} by the matrix
$\begin{pmatrix}1&0\\0&\lambda/\kappa^2\end{pmatrix}$ on the left to obtain
new matrices
$$
A_5=\begin{pmatrix}1&0\\0&\lambda/\kappa^2\end{pmatrix},\qquad
B_5=\begin{pmatrix}0&1\\(\lambda-1)(\lambda-\kappa^2)/(4\kappa^2)&0\end{pmatrix},\qquad
C_5=\begin{pmatrix}\lambda&0\\0&1\end{pmatrix}.
$$

At the \emph{second step} we use the matrix
$L=\begin{pmatrix}2\kappa/(\lambda-\kappa^2)&0\\0&1\end{pmatrix}$ to pass
from $A_5$, $B_5$ and $C_5$ to $A_6=LA_5L^{-1}$, $B_6=L B_5 L^{-1}$ and
$C_6=L C_5 L^{-1}$. Next, multiplying $A_6$, $B_6$ and $C_6$ on the left to
$\begin{pmatrix}\kappa/(\kappa^2-\lambda)&0\\0&\kappa/(1-\lambda)\end{pmatrix}$,
we obtain the matrices
\begin{align*}
A_7=&\begin{pmatrix}\kappa/(\kappa^2-\lambda)&0\\0&\lambda/(\kappa(1-\lambda))\end{pmatrix},\\
B_7=&\begin{pmatrix}0&-1\\-1&0\end{pmatrix},\\
C_7=&\begin{pmatrix}\lambda\kappa/(\kappa^2-\lambda)&0\\0&\kappa/(1-\lambda)\end{pmatrix}.
\end{align*}

Finally, at the \emph{third step} it remains to pass from the system of
equations (on $u$ and $v$) given by $A_7$, $B_7$ and $C_7$, to a single
equation on the function $f=u+iv$. To do this we need to add first equation
of this system with the second one multiplied by the imaginary unit $i$, and
then replace derivatives by $x$ and $y$ with their representations via
$\overline\partial$ and $\partial$, while $u$, $v$ by their representations
via $f$ and $\overline f$.

The passage from the equation \eqref{eq:compl-canon-1} to the form $\mathcal
L_{\tau,\sigma}f=0$ can be checked by the direct computation.
\end{proof}

Notice that the operator $\mathcal L_{\tau,\sigma}$ is more convenient to
represent in the following form. Define the differential operator
$\partial_\tau=\overline\partial+\tau\partial$ and the linear operator
$\mathcal B_{\alpha,\beta}=\alpha\mathcal I+\beta\mathcal C$, where
$\alpha,\beta\in\mathbb R$, while $\mathcal I$ and $\mathcal C$ are the
identity operator and the complex conjugation operator, respectively. Then
$$
\mathcal L_{\tau,\sigma}=\left\{\begin{array}{ll}
\partial\kern1pt\mathcal B_{1,\sigma}\partial_{\tau}&\text{for}\ |\sigma|<1,\\[1em]
\overline\partial\kern1pt\mathcal B_{1,s}\partial_{\tau}&\text{for}\ |\sigma|>1,
\end{array}\right.
$$
where, as previously, $s=1/\sigma$. Moreover, the equation $\mathcal
L_{\tau,\sigma}f=0$ with $\sigma\neq\infty$ may be written as a system with
the operator $\mathcal L$ given by the matrices
\begin{align}
A=&(1+\tau)\begin{pmatrix}1+\sigma&0\\0&1-\sigma\end{pmatrix},\notag\\
B=&\begin{pmatrix}0&\tau-\sigma\\-(\tau+\sigma)&0\end{pmatrix}\label{eq:x1},\\
C=&(1-\tau)\begin{pmatrix}1-\sigma&0\\0&1+\sigma\end{pmatrix}\notag.
\end{align}

\section{Energy functional for the system \eqref{eq:gen-syst}}

For a function $f=u+iv$ in the complex variable $z=x+iy$ we put
\begin{equation}\label{eq:nabla}
\nabla f=\begin{pmatrix} u_x & v_x & u_y & v_y\end{pmatrix}^t.
\end{equation}
Moreover we will identify $f$ with the vector $(u,v)^t$, and it will be
convenient to use the notation $f_x=(u_x,v_x)^t$ and $f_y=(u_y,v_y)^t$. For
vectors $a,b\in\mathbb R^m$, $m\geqslant1$, the symbol $(a,b)$ means, as
usual, their scalar product (in an appropriate space $\mathbb R^m$). The
symbol $m_2(\cdot)$ stands for the two-dimensional Lebesgue measure (i.e.
area) in $\mathbb R^2$.

In this section we study the question on what conditions on the operator
$\mathcal L$ of the form \eqref{eq:gen-oper} ensure that the system
\eqref{eq:gen-syst} with this operator admits a non-negatively determined
energy functional of the form \eqref{eq:enfunc-0}. We start with several
auxiliary statements. Let $E\in M_4(\mathbb R)$ be a symmetric matrix, i.e.
$E=E^t$, and let
\begin{equation}\label{eq:matr-E}
E=\begin{pmatrix}K&L\\L^t&M\end{pmatrix},
\end{equation}
where $K,L,M\in M_2(\mathbb R)$ are such that $K=K^t$ and $M=M^t$.

\begin{lemma}\label{lem:enfunc}
Let $D$ be a Jordan domain in $\mathbb C$ with the boundary $\varGamma$, let
$h\in C(\varGamma)$, and $E\in M_4(\mathbb R)$ be a symmetric matrix. Then,
the system of Euler--Lagrange equations to the functional
\begin{equation}\label{eq:enfunc}
\mathcal Ef:=\frac12\int_D (E\nabla{f},\nabla{f})\,dm_2,
\end{equation}
defined on the class of functions
\begin{equation}\label{eq:class}
\mathcal F(D,h)=\{f\in C^2(D)\cap C^1(\overline{D}) : f|_\varGamma=h\},
\end{equation}
has the form \eqref{eq:gen-syst} with the matrices $A=K$, $B=(L+L^t)/2$ and
$C=M$, where $K$, $L$ and $M$ are defined according to \eqref{eq:matr-E}.
\end{lemma}

\begin{proof}
Write the functional under consideration in terms of matrices $K$, $L$ and
$M$:
$$
\mathcal Ef=\frac12\int_D\big((Kf_x,f_x)+2(Lf_x,f_y)+(Mf_y,f_y)\big)\,dm_2.
$$
A variation of this function may be computed directly:
$$
\delta\mathcal Ef= \int_D\big((Kf_x,\delta f_x)+(L^tf_y,\delta f_x)+(Lf_x,\delta f_y)+(Mf_y,\delta f_y)\big)\,dm_2,
$$
where $\delta f_x=(\delta u_x,\delta v_x)^t$ and $\delta f_y=(\delta
u_y,\delta v_y)^t$ are variations of $f_x$ and $f_y$, respectively. The
latter expression can be reduced to the form
\begin{align*}
\delta\mathcal Ef=&
\int_D\big[\partial_x\big((Kf_x,\delta f)+(L^tf_y,\delta f)\big)+
\partial_y\big((Lf_x,\delta f)+(Mf_y,\delta f)\big)\big]\,dm_2\\
&-\int_D\big[(Kf_{xx},\delta f)+((L+L^t)f_{xy},\delta f)+(Mf_{yy},\delta f)\big]\,dm_2,
\end{align*}
where $\delta f=(\delta u,\delta v)$ is the variation of $f$, and
$\partial_x$ and $\partial_y$ are operator of partial differentiation by $x$
and $y$, respectively. The integral in this expression is a divergence of
some vector and, hence, it is equal to the following inegral
$$
\int_\varGamma\big((Kf_x,\delta f)+(L^tf_y,\delta f)\big)\,dy -\big((Lf_x,\delta f)+(Mf_y,\delta f)\big)\,dx,
$$
which vanishes because the variations of $u$ and $v$ on $\varGamma$ equal
zero (since these functions take on $\varGamma$ the prescribed values). The
equality to zero of the second integral from the expression for
$\delta\mathcal Ef$ leads to the system of the form \eqref{eq:gen-syst}
generated by the matrices $A=K$, $B=(L+L^t)/2$ and $C=M$.
\end{proof}

The next lemma may be proved by direct differentiation.
\begin{lemma}\label{lem:enfunc-trans}
Let $D$ be a Jordan domain with the boundary $\varGamma$ and let $h\in
C(\varGamma)$ be a given function. Moreover, let $E\in M_4(\mathbb R)$ be a
symmetric matrix and a functional $\mathcal E$ on the set $\mathcal F(D,h)$
is defined by \eqref{eq:enfunc}. Then the following properties take place.
\begin{enumerate}
\item[1)]
Let $z\mapsto\zeta=\xi+i\eta$ be the linear non-degenerate change of
variables in $\mathbb R^2$, defined by the matrix $T\in M_2(\mathbb R)$.
Then
$$
\mathcal Ef= \frac{1}{2}\int_{D_0}(E_0\nabla_\zeta f,\nabla_\zeta f)\,dm_2(\zeta),
$$
where $D_0$ is the image of $D$ under the mapping $z\mapsto\zeta$, while
$\nabla_\zeta f=(u_\xi,v_\xi,u_\eta,v_\eta)^t$, and where components
$K_0$, $L_0$ and $M_0$ of $E_0$ are determined from the equality
\begin{equation}\label{eq:vtrans}
\begin{pmatrix}K_0&L_0\\L_0^t&M_0\end{pmatrix}
=T\begin{pmatrix}K&L\\L^t&M\end{pmatrix}T^t.
\end{equation}
\item[2)]
Let $f=u+iv$ and $f_1=u_1+iv_1$ are connected by some non-degenerate
linear transformation $f_1=Qf$, defined by $Q\in M_2(\mathbb R)$. Then
$$
\mathcal Ef=\frac12\int_{D}(E_1\nabla f_1,\nabla f_1)\,dm_2,
$$
where $E_1\in M_4(\mathbb R)$ is such that their components $K_1$, $L_1$
and $M_1$ of $E_1$ are determined from the equality
\begin{equation}\label{eq:ftrans}
\begin{pmatrix}K_1&L_1\\L_1^t&M_1\end{pmatrix}=
Q^t\begin{pmatrix}K&L\\L^t&M\end{pmatrix}Q.
\end{equation}
\end{enumerate}
In the case where $E$ is non-negatively \textup(resp., positively\textup)
determined, that both matrices $E_0$ and $E_1$ also are non-negatively
\textup(resp., positively\textup) determined.
\end{lemma}

The main result of this paper is the following statement:
\begin{theorem}\label{thm:main}
A non-reducible elliptic system of the form \eqref{eq:gen-syst} is the system
of Euler--Lagrange equations for some functional of the form
\eqref{eq:enfunc} with non-negatively determined matrix $E\in M_4(\mathbb R)$
if and only if this system is strongly elliptic and its canonical parameters
$\tau$ and $\sigma$ are such that $\sigma>\tau$.
\end{theorem}

\begin{proof}
Suppose that the system \eqref{eq:gen-syst} is the system of Euler--Lagrange
equations for some functional \eqref{eq:enfunc}. Then, according to
Lemma~\ref{lem:enfunc}, this system has symmetric matrices, or it can be
reduced to the system with symmetric matrices by suitable linear combinations
of equations. The system with symmetric matrices we reduce to the system
generated by matrices of the form \eqref{eq:x1}. It can be done using linear
combinations of equations which do not change the energy functional and by
linear changes of variables and sought-for functions. The latter
transformation, according to Lemma~\ref{lem:enfunc-trans}, preserves the
property of energy functional to be non-negatively determined. Now we
multiply all matrices of the obtained system on the left by the matrix
$$
\begin{pmatrix}\sigma+\tau & 0 \\ 0 & \sigma-\tau\end{pmatrix}
$$
and arrive at the system with symmetric matrices
\begin{equation}\label{eq:matr-sym}
\begin{aligned}
A=&(1+\tau)\left(\begin{matrix}(1+\sigma)(\sigma+\tau) & 0 \\ 0 & (1-\sigma)(\sigma-\tau)\end{matrix}\right),\\
B=&(\tau^2-\sigma^2)\left(\begin{matrix}0 & 1 \\ 1 & 0\end{matrix}\right),\\
C=&(1-\tau)\left(\begin{matrix}(1-\sigma)(\sigma+\tau) & 0 \\ 0 & (1+\sigma)(\sigma-\tau)\end{matrix}\right).
\end{aligned}
\end{equation}
The system generated by $A$, $B$ and $C$ is also the system of
Euler--Lagrange equations for the functional of the form \eqref{eq:enfunc}
with non-negatively determined matrix $E$ of the form \eqref{eq:matr-E}.
Moreover, according to Lemma~\ref{lem:enfunc} we have $K=A$, $L+L^t=2B$ and
$M=C$ for components of $E$. Using the Silvester criterion we conclude, that
if $E$ is non-negatively determined, then $K$ needs to be non-negatively
determined, which yields $\sigma>\tau$ (the case $\sigma=\tau$ is already
excluded since the initial system is non-reducible).

In order to prove sufficiency of our conditions let us present, for
$0\leqslant\tau<\sigma<1$, some non-negatively determined matrix $E$ of the
form \eqref{eq:matr-E} constructed by $K=A$, $M=C$ and
$$
L=\left(\begin{matrix} 0 & (1+\sigma)(\tau^2-\sigma^2) \\ (1-\sigma)(\tau^2-\sigma^2) & 0\end{matrix}\right),
$$
so that $L$ possesses the condition $L+L^t=2B$, where $A$, $B$ and $C$ are
taken from \eqref{eq:matr-sym}.
\end{proof}

In particular, it follows from Theorem~\ref{thm:main} that for systems
determined by operators $\mathcal L_{\tau,0}$ for $\tau>0$ (that is for
equations of the form \eqref{eq:gen-compl-equation} different from the
Laplace equation, which corresponds to a reducible system) do not exist
non-negatively determined energy functional of the form \eqref{eq:enfunc}.
This circumstance shows that it is not possible to generalize directly the
Lebesgue theorem stated above to strongly elliptic equations of the form
\eqref{eq:gen-compl-equation}. Thus, for proving an analogue of the Lebesgue
theorem for such equations (and, in particular, for solving Problem~4.2 from
\cite{pf1999}) some essentially different techniques and ideas are needed.

\end{document}